\documentclass[12pt]{article}
\usepackage{cmap}
\usepackage[T2A]{fontenc}
\usepackage[utf8]{inputenc}
\usepackage[english, ukrainian]{babel}
\usepackage{amssymb,amsthm,a4,graphpap, latexsym,amscd,pb-diagram}
\usepackage[left=3cm, right=2cm, top=2.5cm, bottom=3.5cm, bindingoffset=0cm]{geometry}
\usepackage{amsmath}
\usepackage{amsfonts}
\usepackage{hyperref}
\usepackage{bbm}
\everymath{\displaystyle}

\usepackage[usenames]{color}
\usepackage{colortbl}

\usepackage{graphicx}
\graphicspath{{pictures/}}
\DeclareGraphicsExtensions{.pdf,.png,.jpg}

\newtheorem{theorem}{Теорема}[section]

\theoremstyle{definition}

\newtheorem{remark}{Зауваження}[section]

\newtheorem{definition}{Означення}

\DeclareMathOperator{\R}{\text{$\mathbb{R}$}\;\!}

\DeclareMathOperator{\F}{\text{$\mathfrak{F}$}\;\!}

\DeclareMathOperator{\B}{\text{$\mathfrak{B}$}\;\!}

\DeclareMathOperator{\al}{\text{$\alpha$}\;\!}

\author{Селютін Дмитро}
\title{Про диференціювання відносно фільтрів}
\date{}
\everymath{\displaystyle}

\newcommand{\Addresses}{{
		\bigskip
		\footnotesize
		
		Д.~Селютін, \textsc{Кафедра фундаментальної математики Харківського національного університету імені В.Н. Каразіна}\par\nopagebreak
		\textit{Електронна пошта}: \texttt{d.seliutin@karazin.ua}
		
}}
\begin{document}
	\large
	\pagestyle{empty}
	\maketitle
	
	\begin{abstract}
		Дана коротка стаття містить побудову конструкції, яка узагальнює поняття похідної функції однієї змінної, із використанням теорії фільтрів. В роботі наведено нове поняття, продемонстровано, що воно дійсно узагальнює раніше відому концепцію похідної. Отримано властивості та правила диференціювання аналогічно до класичних правил диференціювання. Показано переваги нового поняття перед класичної похідною.
	\end{abstract}
	
	\section{Вступ}
	
	Нагадування: нехай $f: \R \rightarrow \R$ -- функція, неперервна в точці $x_0 \in \R$. Відомо, що похідною функції $f(x)$ в точці $x_0$ називають
	$$
	\frac{df}{dx}(x_0) = \lim\limits_{h \rightarrow 0} \frac{f(x_0+h) - f(x_0)}{h}.
	$$
	
	Широко відомі умови існування похідної, критерій диференційовності, правила диференціювання, таблиця похідних тощо. Дане поняття є одним із центральних (поруч із інтегралом) концепцій математичного аналізу. Похідна має широкий спектр застосувань не тільки в математичній науці (теорія функцій, теорія ймовірностей, диференціальна геометрія тощо), проте займає одну іх провідних ролей в прикладних дисциплінах: різноманітні розділи фізики, біологія, економіка та економічна статистика, наука про дані. Проте поруч із величезним спектром застосувань, звичайна похідна функції має кілька недоліків: по-перше, функція має бути неперервною в точці, по-друге, ''модулеподібні'' не мають похідної у своїх ''гострих'' точках. Наприклад, добре відомо, що функція $f(x) = |x|$ не диференційовна в точці $x_0 = 0$ в силу того, що односторонні похідні даної функції в даній точці є різними. Звісно, одностороння похідна певним чином вирішує цю проблему, проте для її застосування потрібні два різні  поняття ''правої'' та ''лівої'' похідних. Так, ці поняття дуже схожі та мають косметичні відмінності, проте це два різних поняття. В нещодавній статті \cite{sel_int} нами було запропоновано поняття визначеного інтеграла відносно фільтра. Стаття містить докладний опис схеми побудови інтеграла відносно фільтра, вивчено його властивості та переваги, порівняно зі звичайним інтегралом Рімана по відрізку. Визначення поняття похідної в термінах фільтрах -- це абсолютно логічний та природній крок, оскільки і похідна функції, і інтеграл Рімана -- концепції, які для своєї побудови вимагають ''граничного переходу''. Тому в наступних розділах даної короткої статті ми опишемо загальну конструкції похідної функції із використанням техніки фільтрів. Після цього ми порівняємо цю нову концепцію із відомим визначенням похідної і переконаємося в тому, що наша конструкція дійсно є узагальненням. Далі ми будемо вивчати властивості похідної за фільтром, зокрема, відповімо на питання: чи працює для похідної за фільтром правило диференціювання добутку?
	
	\section{Необхідні відомості із теорії фільтрів}
	
	Нагадаємо, що фільтром на непорожній множині $\Omega$ називають таку сім'ю підмножин $\F \subset 2^{\Omega}$, яка задовольняє наступним аксіомам:
	\begin{enumerate}
		\item $\emptyset \notin \F$;
		\item $A, B \in \F \Rightarrow A\cap B \F$;
		\item $D \subset \Omega,\ A \supset D \Rightarrow D \in \F$.
	\end{enumerate}
	
	Поруч із поняттям фільтра ми також будемо використовувати базу фільтра. Сім'ю $\B$ підмножин множини $\Omega$ називають базою фільтра, якщо 
	\begin{enumerate}
		\item $\emptyset \notin \B$;
		\item $\forall A, B \in \B \exists C \in B:\ C \subset A \cap B$.
	\end{enumerate}
	
	Говорять, що фільтр $\F$ породжено базою $\B$, якщо кожен елемент фільтра $\F$  містить хоча б один елемент бази фільтра $\B$. Через $\F(\B)$ будемо позначати фільтр, породжений базою $\B$.
	
	Останнє визначення, яке нам варто пригадати - -це визначення збіжності функції за фільтром. Отож нехай $X,Y$ -- топологічні простори, $f: X \rightarrow Y$ -- функція, $y  \in X$,  $\F$ -- фільтр на $X$. Говорять, що  функція $f$ збігається до $y$ за фільтром $\F$ (позначення: $y = \lim\limits_{\F} f$), якщо для довільного околу $U$ точки $y$ існує елемент фільтра $A \in \F$ такий, що $\{f(t):t \in A\} \subset U$.Більше про теорію можна прочитати, наприклад. у чудовому підручнику \cite{Kad}.
	
	\section{Похідна функції за фільтром}
	
	Нехай $f: \R \rightarrow \R$ -- функція, визначена в деякому околі точки $x_0 \in \R$. Розглянемо функцію:
	$$
	D_{x_0}(f,\cdot):\R \rightarrow \R,
	$$
	
	яка діє за наступним правилом:
	
	$$
	D_{x_0}(f,h) = \frac{1}{h} \left(f(x_0+h) - f(x_0)\right).
	$$
	
	Очевидно, що 
	
	$$
	\frac{df}{dx}(x_0) = \lim\limits_{h \rightarrow 0} D_{x_0}(f,h).
	$$
	
	Тепер ми готові дати центральне визначення даної статті.
	
	\begin{definition}\label{main_def}
		Нехай $f: \R \rightarrow \R$ -- функція, визначена в деякому околі точки $x_0 \in \R$. Нехай $\F$ -- фільтр на $\R$. \textit{Похідною функції $y=f(x)$ в точці $x_0$ відносно фільтра $\F$} будемо називати 
		\begin{equation}
			\frac{df}{d\F}(x_0) = \lim\limits_{\F} D_{x_0}(f,h),
		\end{equation}
	\end{definition}
	якщо дана границя існує.
	
	\begin{remark}
		Функції, у яких є похідна в даній точці за певним фільтром, будемо називати диференційовними за цим фільтром в цій точці.
	\end{remark}
	
	Покажемо тепер, що класичне визначення поняття похідної є частковим випадком Означення \ref{main_def}. Дійсно, для довільного $\delta > 0$ розглянемо $B_{\delta} = (x_0-\delta, x_0+\delta)\setminus {x_0}$. Нескладно перевірити, що сім'я множин $\B_{x_0} = (B_{\delta})_{\delta > 0}$ утворює базу фільтра. Позначимо
	
	$$
	\F_{x_0} = \F(\B_{x_0}).
	$$
	
	Очевидно, що в цьому випадку 
	$$
	\frac{df}{dx}(x_0) = \lim\limits_{\F_{x_0}} D_{x_0}(f,h).
	$$
	
	\section{Властивості похідної функції за фільтром}
	
	Для ''класичної'' похідної можна сформулювати добре відомі усім властивості: вона задовольняє властивостям лінійності та однорідності, визначено правила диференціювання добутку і частки двох функцій. В даному розділі ми будемо вивчати узагальнення цих властивостей для похідної функції відносно фільтра.
	
	\begin{theorem}
		Нехай $f,\ g: \R \rightarrow \R$ -- функції, визначені в деякому околі точки $x_0 \in \R$. Нехай $\F$ -- фільтр на $\R$, $\al, \beta \in \R$. Тоді
		\begin{equation}
			\frac{d(\al f+\beta g)}{d\F}(x_0) = \al \frac{df}{d\F}(x_0)+\beta \frac{dg}{d\F}(x_0),
		\end{equation}
	іншими словами, похідна за фільтром від суми двох функцій -- це сума похідних за фільтром від цих двох функцій.
	\end{theorem}  
	Доведення цього факту спирається на застосування властивості лінійності та однорідності границі функції за фільтром (див. \cite{sel_int}). Краще ми зосередимо свою увагу на дослідженні правил диференціювання добутку і частки.
	
	Зараз ми дамо одне технічне визначення.
	
	\begin{definition}
		Нехай $f: \R \rightarrow \R$ -- функція, визначена в деякому околі точки $a \in \R$, $\F$ -- фільтр на $\R$. Будемо називати функцію $f$ \textit{$\F$-неперервною в точці $a$}, якщо $\lim\limits_{\F} f(a+h) = f(a)$.  
	\end{definition}
	
	\begin{theorem}
		Нехай $f,\ g: \R \rightarrow \R$ -- функції, визначені в деякому околі точки $x_0 \in \R$, $\F$ -- фільтр на $\R$. Нехай існують похідні функцій $f,\ g$ відносно фільтра $\F$ в точці $x_0$. Також нехай функції $f$ та  $g$ є $\F$-неперервними в точці $x_0$. Тоді $f \cdot g$ також є диференційовною в точці $x_0$ за фільтром $\F$, і 
		\begin{equation}
			\frac{d( f \cdot g)}{d\F}(x_0) = \frac{df}{d\F}(x_0) \cdot g(x_0) + \frac{dg}{d\F}(x_0) \cdot f(x_0).
		\end{equation}
	\end{theorem}  
	
	\begin{proof}
		Дійсно, 
				$$
				\frac{d( f \cdot g)}{d\F}(x_0) = \lim\limits_{\F} D_{x_0}(f\cdot g,h) = \lim\limits_{\F} \frac{1}{h} \left((f\cdot g)(x_0+h) - (f\cdot g)(x_0)\right) = 
				$$
				$$
				\lim\limits_{\F} \frac{1}{h} \left(f(x_0+h)\cdot g(x_0+h) - f(x_0)\cdot g(x_0)\right) = 
				$$
				\begin{align*}
				\lim\limits_{\F} \frac{1}{h} \left(f(x_0+h)\cdot g(x_0+h) + f(x_0)\cdot g(x_0+h) \right. \\
				\left. - f(x_0)\cdot g(x_0+h) - f(x_0)\cdot g(x_0)\right) = \\
				\lim\limits_{\F} \frac{1}{h} \bigg[g(x_0+h)\big(f(x_0+h) - f(x_0)\big) + f(x_0)\big(g(x_0+h) - g(x_0)\big)\bigg] = \\
				\lim\limits_{\F} \bigg[ \frac{g(x_0+h)\big(f(x_0+h) - f(x_0)\big)}{h} + \frac{f(x_0)\big(g(x_0+h) - g(x_0)\big)}{h}\bigg] = \\
				\lim\limits_{\F} \bigg[ \frac{g(x_0+h)\big(f(x_0+h) - f(x_0)\big)}{h}\bigg] + \lim\limits_{\F} \bigg[\frac{f(x_0)\big(g(x_0+h) - g(x_0)\big)}{h}\bigg] = \\
				\lim\limits_{\F} g(x_0+h)\cdot \lim\limits_{\F} \bigg[ \frac{\big(f(x_0+h) - f(x_0)\big)}{h}\bigg] +  f(x_0)\cdot \lim\limits_{\F} \bigg[ \frac{\big(g(x_0+h) - g(x_0)\big)}{h}\bigg] = \\
				\frac{df}{d\F}(x_0) \cdot g(x_0) + \frac{dg}{d\F}(x_0) \cdot f(x_0).
				\end{align*}
	\end{proof}
	
	\begin{remark}
		Навіщо ми ввели поняття $\F$-неперервності? Для того, аби у третьому знизу рядку цієї теореми мати змогу перейти до добутку границь. В умові теореми ми вимагали $\F$-неперервності від обох функцій, проте фактично використали цю особливості лише для функції $g$. Ми вимагаємо того самого і від функції $f$, оскільки в доведенні ми можемо додати і відняти інший доданок, а саме $g(x_0)\cdot f(x_0+h)$.
	\end{remark}
	
	Переходимо тепер до правила диференціювання за фільтри частки функцій.
	
	\begin{theorem}
		Нехай $f,\ g: \R \rightarrow \R$ -- функції, визначені в деякому околі точки $x_0 \in \R$, $\F$ -- фільтр на $\R$. Нехай існують похідні функцій $f,\ g$ відносно фільтра $\F$ в точці $x_0$. Також нехай функція $g$ є $\F$-неперервною в точці $x_0$, $g(x_0) \neq 0$. Тоді частка функцій $\bigg(f/g\bigg)$ також є диференційовною в точці $x_0$ за фільтром $\F$, і 
		\begin{equation}
			\frac{d \bigg(f/g\bigg)}{d\F}(x_0) = \frac{\frac{df}{d\F}(x_0)\cdot g(x_0) - \frac{dg}{d\F}(x_0)\cdot g(x_0)}{g^2(x_0)}
		\end{equation}
	\end{theorem}  
	
	\begin{proof}
		Дійсно,
			\begin{align*}
			\frac{d \bigg(f/g\bigg)}{d\F}(x_0) = 	\lim\limits_{\F} \frac{1}{h}\bigg[\left(\frac{f}{g}\right)(x_0+h) - \left(\frac{f}{g}\right)(x_0)\bigg] = \\
			\lim\limits_{\F} \frac{1}{h}\bigg[\frac{f(x_0+h)}{g(x_0+h)} - \frac{f(x_0)}{g(x_0)}\bigg] = \lim\limits_{\F} \frac{1}{h}\bigg[\frac{g(x_0)\cdot f(x_0+h) - f(x_0)\cdot g(x_0+h))}{g(x_0)\cdot g(x_0+h)}\bigg].\\
			\end{align*}
		В чисельнику додамо та віднімемо вираз $g(x_0)\cdot f(x_0)$. Отримаємо
			\begin{align*}
			\lim\limits_{\F} \frac{1}{h}\bigg[\frac{g(x_0)\cdot f(x_0+h) - g(x_0)\cdot f(x_0) + g(x_0)\cdot f(x_0)- f(x_0)\cdot g(x_0+h))}{g(x_0)\cdot g(x_0+h)}\bigg] = \\
			\lim\limits_{\F} \frac{1}{h}\bigg[\frac{g(x_0)\big(f(x_0+h) - f(x_0)\big) - f(x_0)\big(g(x_0+h)-g(x_0)\big)}{g(x_0)\cdot g(x_0+h)}\bigg] = \\
			\lim\limits_{\F} \bigg[\frac{g(x_0)\big(f(x_0+h) - f(x_0)\big)}{h \cdot g(x_0)\cdot g(x_0+h)} -\frac{ f(x_0)\big(g(x_0+h)-g(x_0)\big)}{h \cdot  g(x_0)\cdot g(x_0+h)}\bigg] = \\
			\lim\limits_{\F} \bigg[\frac{g(x_0)\cdot \frac{f(x_0+h) - f(x_0)}{h}}{g(x_0)\cdot g(x_0+h)} -\frac{ f(x_0)\cdot \frac{g(x_0+h) - g(x_0)}{h}}{g(x_0)\cdot g(x_0+h)}\bigg].
		\end{align*}
	Оскільки функція $g$ є $\F$-неперервною в точці $x_0$, маємо:
	\begin{align*}
		\lim\limits_{\F} \frac{g(x_0)\cdot \frac{f(x_0+h) - f(x_0)}{h}}{g(x_0)\cdot g(x_0)}  - 	\lim\limits_{\F} \frac{ f(x_0)\cdot \frac{g(x_0+h) - g(x_0)}{h}}{g(x_0)\cdot g(x_0)}=\\
		\frac{g(x_0)\cdot \frac{df}{d\F}(x_0)}{g^2(x_0)}  - 	\frac{ f(x_0)\cdot \frac{dg}{d\F}(x_0)}{g^2(x_0)}=\\
		\frac{\frac{df}{d\F}(x_0)\cdot g(x_0) - \frac{dg}{d\F}(x_0)\cdot g(x_0)}{g^2(x_0)}
	\end{align*}
	\end{proof}
	
	\begin{remark}
		Таким чином, ми бачимо, що для похідної функції відносно фільтра мають місце всі властивості, які були притаманні класичній похідній.
	\end{remark}
	
	Розглянемо ще один пункт, який ми трохи згадали у вступній частині: диференціювання функції $f(x) = |x|$ в точці $x_0 = 0$. Так, в цій точці модуль не має похідної у звичайному розумінні цього слова, проте має односторонні похідні. Проте тепер, маючи такий інструмент, як похідна функції за фільтром, обчислення похідної даної функції в даній точці стає набагато елегантнішим. Дійсно, розглянемо сім'ї множин
	$$
	\B_{0+} = ((0,\delta))_{\delta>0},\ \B_{0-} = ((-\delta,0))_{\delta>0}.
	$$
	
	Очевидно, що ці сім'ї утворюють бази фільтрів.
	
	Тоді
	$$
	\frac{df}{d\F(\B_{0+} )}(0) = 1,\ \frac{df}{d\F(\B_{0-} )}(0) = -1.
	$$
	
	\vspace{15mm}
	
	\textbf{Подяки\\}
	Статтю підготовлено в рамках виконання держбюджетної теми Міністерства освіти і науки України № БФ/32-2021(11). Також автор висловлює безмежні слова вдячності своїм батькам за постійну підтримку. Окремі слова вдячності хочеться висловити Силам Безпеки та Оборони України, які в даний момент дають рішучу відсіч російським окупантам.

	\Addresses

\end{document}